\documentclass[12pt]{article}
\usepackage[affil-it]{authblk}
\usepackage[left=1in,right=1in,top=1.2in, bottom=1in]{geometry}
\usepackage{amsmath,amsthm,amssymb,amsfonts,pgf,tikz,url} % needed for DeclareMathOperator
\usepackage{bm} % needed for bold math

\newtheorem{theorem}{Theorem}[section]
\newtheorem{lemma}[theorem]{Lemma}
\newtheorem{corollary}[theorem]{Corollary}
\newtheorem{proposition}[theorem]{Proposition}

\theoremstyle{definition}

\newtheorem{example}[theorem]{Example}
\newtheorem{obs}[theorem]{Observation}
\newtheorem{question}[theorem]{Question}

\theoremstyle{remark}

\DeclareMathOperator{\rank}{rank}

\renewcommand{\i}{\rm i}

\begin{document}
\title{Matrix tree theorem for the net Laplacian matrix of a signed graph}
\author{Sudipta Mallik}
\affil{\small Department of Mathematics and Statistics, Northern Arizona University, 801 S. Osborne Dr.\\ PO Box: 5717, Flagstaff, AZ 86011, USA 

sudipta.mallik@nau.edu}

\maketitle
\begin{abstract}
For a simple signed graph $G$ with the adjacency matrix $A$ and net degree matrix $D^{\pm}$, the net Laplacian matrix is $L^{\pm}=D^{\pm}-A$. We introduce a new oriented incidence matrix $N^{\pm}$ which can keep track of the sign as well as the orientation of each edge of $G$. Also  $L^{\pm}=N^{\pm}(N^{\pm})^T$. Using this decomposition, we find the number of both positive and negative spanning trees of $G$ in terms of the principal minors of $L^{\pm}$ generalizing the Matrix Tree Theorem for an unsigned graph. We present similar results for the signless net Laplacian matrix $Q^{\pm}=D^{\pm}+A$ along with a combinatorial formula for its determinant.
\end{abstract}

%Key words: Incidence matrix, Signed graph, Net Laplacian matrix, Matrix tree theorem
%MSC2020 classifications 05C50; 05C22; 15B99
%%05C50 Graphs and linear algebra (matrices, eigenvalues, etc.), 05C22 Signed and weighted graphs, 15B99 None of the above, but in this section (15Bxx Special matrices) 

\section{Introduction}
A {\it signed graph} $G = \{V, E, \sigma \}$ is defined to be a graph on vertex set $V = \{v_1,\ldots,v_n\}$ and  edge set $E = \{e_1,\ldots,e_m\}$, where each edge $e_{\ell}$ has a sign $\sigma(e_{\ell}) \in \{1,-1\}$. An edge with a positive sign is a {\it positive edge} and an edge with a negative sign is a {\it negative edge}.  Throughout this article, we consider only simple signed graphs, i.e., signed graphs that do not have loops and multiple edges. The {\it adjacency matrix} of signed graph $G$ is an $n\times n$ binary matrix $A=[a_{ij}]$ where $a_{ij}$ is $\sigma\{i,j\}$ if $\{i,j\}$ is an edge and $0$ otherwise. Note that $|A|$, obtained from $A$ by taking absolute values of its entries, is the adjacency matrix of the underlying unsigned graph $|G|$ of the signed graph $G$. Note that $|G|$ can be thought as the signed graph obtained from $G$ by making the sign of each edge of $G$ positive.  The {\it degree matrix} of $G$ and $|G|$, denoted by $D$, is an $n\times n$ diagonal matrix with the degrees of vertices  of $G$ as diagonal entries. The {\it net degree} of a vertex $v$, denoted by $d^{\pm}(v)$, in signed graph $G$ is the number of positive edges minus that of negative edges incident with $v$.  The {\it net degree matrix} of $G$, denoted by $D^{\pm}$, is an $n\times n$ diagonal matrix with the net-degrees of vertices  of $G$ as diagonal entries.  The {\it Laplacian}, {\it signless Laplacian}, {\it net Laplacian}, and {\it signless net Laplacian}  of signed graph $G$ are defined as $L=D-A$, $Q=D+A$, $L^{\pm}=D^{\pm}-A$, and $Q^{\pm}=D^{\pm}+A$ respectively. Note that the entries of the signless Laplacian $Q$ and the signless net Laplacian $Q^{\pm}$ are not necessarily signless. The author could not come up with better names for $Q$ and $Q^{\pm}$.

\begin{example}
For the signed paw  $G$  in Figure \ref{paw},  \[L=D-A=\left[\begin{array}{rrrr}
    1&-1&0&0\\
	-1&3&1&-1\\
	0&1&2&1\\
	0&-1&1&2
\end{array} \right],
Q=D+A=\left[\begin{array}{rrrr}
    1&1&0&0\\
	1&3&-1&1\\
	0&-1&2&-1\\
	0&1&-1&2
\end{array} \right],\]
\[L^{\pm}=D^{\pm}-A=\left[\begin{array}{rrrr}
    1&-1&0&0\\
	-1&1&1&-1\\
	0&1&-2&1\\
	0&-1&1&0
\end{array} \right],
Q^{\pm}=D^{\pm}+A=\left[\begin{array}{rrrr}
    1&1&0&0\\
	1&1&-1&1\\
	0&-1&-2&-1\\
	0&1&-1&0
\end{array} \right].
\]

\begin{figure}
	\begin{center}
	\begin{tikzpicture}[scale=1.1, colorstyle/.style={circle, fill, black, scale = .5}, >=stealth]
		
		\node (1) at (0,1)[colorstyle, label=above:$1$]{};
		\node (2) at (0,0)[colorstyle, label=below:$2$]{};
		\node (3) at (-1,-1)[colorstyle, label=below:$3$]{};
		\node (4) at (1,-1)[colorstyle, label=below:$4$]{};	
		
		\draw [] (1)--(2)--(3)--(4)--(2);
		\node at (0.85,0)[label=below:{$+e_4$}]{};
		\node at (-0.95,0)[label=below:{$-e_2$}]{};
		\node at (0,-.8)[label=below:{$-e_3$}]{};
		\node at (-.25,0.5)[label=right:{$+e_1$}]{};
		
		%the right graph
		\node (1b) at (5,1)[colorstyle, label=above:$1$]{};
		\node (2b) at (5,0)[colorstyle, label=below:$2$]{};
		\node (3b) at (4,-1)[colorstyle, label=below:$3$]{};
		\node (4b) at (6,-1)[colorstyle, label=below:$4$]{};	
		
 		\draw[->] (1b) edge (2b)(3b) edge (2b)(3b) edge (4b)(2b) edge (4b);
		\node at (5.85,0)[label=below:{$+e_4$}]{};
		\node at (4.05,0)[label=below:{$-e_2$}]{};
		\node at (5,-.8)[label=below:{$-e_3$}]{};
		\node at (4.75,0.5)[label=right:{$+e_1$}]{};

	\end{tikzpicture}
	\caption{Signed paw $G$ and one of its orientation}\label{paw}
	\end{center}
	\end{figure}
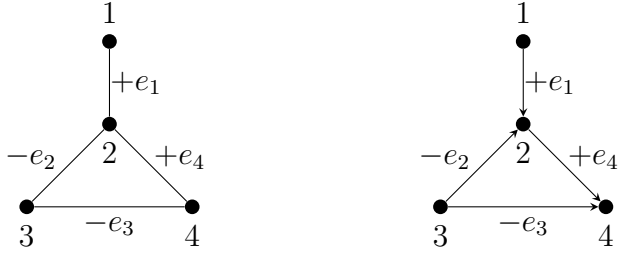
\end{example}

\begin{obs}
Suppose $-G$ is the signed graph obtained from the signed graph $G$ by reversing the sign of each of its edges. The following shows the relations among matrices associated with $G$ and $-G$:
\begin{enumerate}
\item $L_{-G}=D_{-G}-A_{-G}=D_{G}+A_{G}=Q_{G}$.

\item $Q_{-G}=D_{-G}+A_{-G}=D_{G}-A_{G}=L_{G}$.

\item $L_{-G}^{\pm}=D^{\pm}_{-G}-A_{-G}=-D^{\pm}_{G}+A_{G}=-L_{G}^{\pm}$.

\item $Q_{-G}^{\pm}=D^{\pm}_{-G}+A_{-G}=-D^{\pm}_{G}-A_{G}=-Q_{G}^{\pm}$.
\end{enumerate} 
\end{obs}

We define a new incidence matrix of a signed graph $G$ which is called the {\it net incidence  matrix}, denoted by $M^{\pm}=[m_{ij}]$, as follows: $M^{\pm}$ is an $n\times m$  matrix with rows indexed by vertices and columns indexed by edges where each column of $M^{\pm}$ has exactly two nonzero entries  $m_{i \ell}=m_{j \ell}$ equaling to $1$ if  $e_\ell = \{i,j\}$ is a positive edge  and the imaginary number $\i$ if  $e_\ell = \{i,j\}$ is a negative edge.

\begin{example}
For the signed paw  $G$  in Figure \ref{paw},  \[M^{\pm}=\left[\begin{array}{rrrr}
    1&0  &0     &0\\
	1&\i &0     &1\\
	0&\i &\i    &0\\
	0&0  &\i    &1
\end{array} \right].\]
\end{example}

Similarly we define a new oriented incidence matrix of a signed graph $G$ which is called an {\it oriented net incidence  matrix}, denoted by $N^{\pm}=[n_{ij}]$, as follows: $N^{\pm}$ is obtained from the net incidence  matrix of $G$ by multiplying exactly one of the two nonzero entries of each column by $-1$. An  oriented net incidence  matrix $N^{\pm}=[n_{ij}]$ of a signed graph $G$ induces the following orientation of edges: edge $e_\ell = \{i,j\}$ is oriented as $(i,j)$ (i.e., vertex $i$ to vertex $j$) if $m_{i \ell}>m_{j \ell}$ on the real or imaginary axis.  In the literature on signed graph, edges are oriented with two arrows giving a bidirected graph \cite{Belardo1}. When each edge of a signed graph is singly oriented with one arrow, the oriented net incidence matrix can keep track of the sign as well as the orientation of each edge.

\begin{example}
For the oriented signed paw  $G$  in Figure \ref{paw},  \[N^{\pm}=\left[\begin{array}{rrrr}
    1&0   &0     &0\\
   -1&-\i &0     &1\\
	0&\i  &\i    &0\\
	0&0   &-\i   &-1
\end{array} \right].\]
\end{example}

For some literature on signed graphs and associated matrices, see \cite{Zaslavsky,Belardo1,Belardo2}. In \cite{Belardo1}, the coefficients of the characteristic polynomial of the Laplacian matrix of a signed graph were studied. 
The net Laplacian of a signed graph has recently started getting more attention. For example, it was used in the study of controllability of undirected signed graphs \cite{Gao}.
Stani\' c studied the spectrum of the net Laplacian matrix of a signed graph in \cite{Stanic1}. He mentioned the lack of a decomposition for the net Laplacian of a signed graph similar to that of Laplacian as the product of an incidence matrix and its transpose. The net incidence matrix defined above provides such a decomposition which is explained in Section 2. As a consequence of this decomposition,  we obtain a matrix tree theorem analog for the net Laplacian of a signed graph in Section 3. In recent years, different combinatorial aspects of signless Laplacian matrices of graphs have emerged as active areas of research \cite{CRC,Hessert1,Hessert2,Mallik,Ipsen}. In Section 4 we present a matrix tree theorem analog for the signless net Laplacian of a signed graph. In Section 5 we find a combinatorial formula for the determinant of the signless net Laplacian.

\section{Basic results}
In this section we present some basic results regarding net incidence and net Laplacian matrices similar to that of incidence and  Laplacian matrices.

\begin{proposition}
Let $D$, $D^{\pm}$, $A$, $M^{\pm}$, $L^{\pm}$, $Q^{\pm}$, and $Q$ be the degree, net degree, adjacency, net incidence, net Laplacian, signless net Laplacian, and signless Laplacian matrices of a simple signed graph $G$ respectively. Let $L_{|G|}$ and $Q_{|G|}$ be the Laplacian and signless Laplacian matrices  of  the underlying unsigned graph $|G|$ of the signed graph $G$ respectively. 
\begin{enumerate}
\item[(a)] $Q^{\pm}=D^{\pm}+A=M^{\pm}(M^{\pm})^T$.

\item[(b)] $L^{\pm}=D^{\pm}-A=N^{\pm}(N^{\pm})^T$ for any oriented net incidence matrix $N^{\pm}$ of $G$.

\item[(c)] $Q_{|G|}=D+|A|=M^{\pm}(M^{\pm})^*$.

\item[(d)] $L_{|G|}=D-|A|=N^{\pm}(N^{\pm})^*$ for any oriented net incidence matrix $N^{\pm}$ of $G$.
\end{enumerate}
\end{proposition}
\begin{proof}
Suppose $G$ has $n$ vertices. To prove (a), suppose the rows of $M$ are $M_1,M_2,\ldots,M_n$. Then the $(i,j)$-entry of $M^{\pm}(M^{\pm})^T$ is $(M^{\pm}(M^{\pm})^T)_{ij}=M_iM_j^T$. This implies $(M^{\pm}(M^{\pm})^T)_{ii}=d^{\pm}(v_i)$ and for $i\neq j$,   $(M^{\pm}(M^{\pm})^T)_{ij}=0$ when $\{i,j\}$ is not an edge of $G$ and   $(M^{\pm}(M^{\pm})^T)_{ij}=\sigma\{i,j\}$ when $\{i,j\}$ is an edge of $G$. Thus $Q^{\pm}=D^{\pm}+A=M^{\pm}(M^{\pm})^T$.

The proof is similar for (b) by noting that $(N^{\pm}(N^{\pm})^T)_{ij}=-\sigma\{i,j\}$ when $\{i,j\}$ is an edge of $G$. The proofs for (c) and (d) are similar.
\end{proof}

\begin{proposition}
Let $G$ be a simple signed graph on $n$ vertices with the positive edge set $E^+$, negative edge set $E^-$, net Laplacian $L^{\pm}$, and signless net Laplacian $Q^{\pm}$. Then for all $x\in \mathbb R^n$, 
\[x^T L^{\pm}x=\displaystyle\sum_{\{u,v\}\in E^+}(x_u-x_v)^2-\sum_{\{i,j\}\in E^-}(x_i-x_j)^2\]
and
\[x^T Q^{\pm}x=\displaystyle\sum_{\{u,v\}\in E^+}(x_u+x_v)^2-\sum_{\{i,j\}\in E^-}(x_i+x_j)^2.\]
\end{proposition}	
	
\begin{proof}
Suppose $L^{\pm}=N^{\pm}(N^{\pm})^T$ for some oriented net incidence matrix $N^{\pm}$ of $G$. For any $x\in \mathbb R^n$, $(N^{\pm})^T x\in \mathbb C^m$. Each entry of $(N^{\pm})^T x$ corresponds to an edge of $G$ and has the form $\pm (x_u-x_v)$ if $\{u,v\}\in E^+$ and $\pm i(x_i-x_j)$ if $\{i,j\}\in E^-$. Then 
\begin{align*}
x^TLx 	&= x^TN^{\pm} (N^{\pm})^T x\\
		&= ((N^{\pm})^Tx)^T ((N^{\pm})^Tx)\\
		&=\displaystyle\sum_{\{u,v\}\in E^+}[\pm(x_u-x_v)]^2+\sum_{\{i,j\}\in E^-}[\pm i(x_i-x_j)]^2\\
		&=\displaystyle\sum_{\{u,v\}\in E^+}(x_u-x_v)^2-\sum_{\{i,j\}\in E^-}(x_i-x_j)^2.
\end{align*}

A similar proof can derive the result for $Q^{\pm}=M^{\pm}(M^{\pm})^T$ where $M^{\pm}$ is the net incidence matrix of $G$.
\end{proof}	

\begin{obs}\label{obs eigenvector}
If $x\in \mathbb R^n$ is an eigenvector corresponding to the eigenvalue $0$ of the net Laplacian matrix $L^{\pm}$ of a simple signed graph, then 
\[\displaystyle\sum_{\{u,v\}\in E^+}(x_u-x_v)^2=\sum_{\{i,j\}\in E^-}(x_i-x_j)^2.\]
If $0$ is an eigenvalue of the signless net Laplacian matrix $Q^{\pm}$ of a simple signed graph with an eigenvector $x\in \mathbb R^n$, then
\[\displaystyle\sum_{\{u,v\}\in E^+}(x_u+x_v)^2=\sum_{\{i,j\}\in E^-}(x_i+x_j)^2.\]
\end{obs}

\begin{proposition}
Let $M^{\pm}$ and $N^{\pm}$ be the net incidence matrix and an oriented net incidence matrix of a simple signed graph $G$ respectively. Let $M_{|G|}$ and $N_{|G|}$ be the incidence matrix and an oriented incidence matrix of  the underlying unsigned graph $|G|$ of the signed graph $G$ respectively. Then 
\[\rank(M^{\pm})=\rank(M_{|G|}) \text{ and } \rank(N^{\pm})=\rank(N_{|G|}).\]
\end{proposition}

\begin{proof}
The proof follows from that fact that 
the left null spaces of $M^{\pm}$ and $M_{|G|}$ are the same and the left null spaces of $N^{\pm}$ and $N_{|G|}$ are the same.
\end{proof}

\begin{obs} Let $N^{\pm}$, $M^{\pm}$, $L^{\pm}$, $Q^{\pm}$, and $Q$ be  an oriented net incidence, the net incidence, net Laplacian, signless net Laplacian, and signless Laplacian matrices of a simple signed graph $G$ respectively. Let $Q_{|G|}$ be the signless Laplacian matrix  of  the underlying unsigned graph $|G|$ of the signed graph $G$.
\begin{enumerate}
\item[(a)] $\rank(M^{\pm})= \rank(M^{\pm}(M^{\pm})^*)=\rank(Q_{|G|})$.

\item[(b)] $\rank(M^{\pm})\geq \rank(M^{\pm}(M^{\pm})^T)=\rank(Q^{\pm})$.

\item[(c)]
$\rank(N^{\pm})\geq \rank(N^{\pm}(N^{\pm})^T)=\rank(L^{\pm})$.
\end{enumerate}
\end{obs}

\begin{lemma}\label{det of inc of a cycle}
Suppose $G$ is a signed cycle on $n$ vertices. Then  the determinant of an oriented net incidence matrix of $G$ is $0$. Also the determinant of the net incidence matrix of $G$ is $\pm 2 \i^{e^-}$ if $n$ is odd and zero otherwise, where $\mathrm e^-$ is the number of negative edges of $G$.
\end{lemma}
\begin{proof}
Suppose $M^{\pm}$ is the net  incidence matrix of  $G$. Then 
	$$M^{\pm}=PM_*^{\pm}Q=P
	\left[\begin{array}{cccccc}
	a_1&0&0&\cdots&0&a_n\\
	a_1&a_2&0&\cdots&0&0\\
	0&a_2&a_3&\ddots&\vdots&\vdots\\
	\vdots&\ddots &\ddots&\ddots&\ddots&\vdots\\
	\vdots&\vdots&\ddots&\ddots&a_{n-1}&0\\
	0&\cdots&\cdots&0&a_{n-1}&a_n\\
	\end{array}\right]Q,
	$$
for some permutation matrices $P$ and $Q$ and for some $a_1,a_2,\ldots,a_n \in \{1,\i\}$. By a cofactor expansion across the first row, 
\begin{align*}
\det(M^{\pm}) &=\det(P)\det(M_*^{\pm})\det(Q)\\
&=(\pm 1)[a_1a_2 \cdots a_n + (-1)^{n+1} a_na_1\cdot a_2 \cdots a_{n-1}](\pm 1)\\
&=(\pm 1)a_1a_2 \cdots a_n [1+(-1)^{n+1}](\pm 1)\\
&=\pm \i^{e^-}[1+(-1)^{n+1}].   
\end{align*}

Then $\det(M^{\pm})=\pm 2 \i^{e^-}$ if $n$ is odd and zero otherwise.

A similar proof can derive the result for any oriented net incidence matrix $N^{\pm}$:
$$N^{\pm}=PN_*^{\pm}Q=P
	\left[\begin{array}{cccccc}
	a_1&0&0&\cdots&0&-a_n\\
	-a_1&a_2&0&\cdots&0&0\\
	0&-a_2&a_3&\ddots&\vdots&\vdots\\
	\vdots&\ddots &\ddots&\ddots&\ddots&\vdots\\
	\vdots&\vdots&\ddots&\ddots&a_{n-1}&0\\
	0&\cdots&\cdots&0&-a_{n-1}&a_n\\
	\end{array}\right]Q,
	$$
for some permutation matrices $P$ and $Q$ and for some $a_1,a_2,\ldots,a_n \in \{1,\i\}$. By a cofactor expansion across the first row, 
\begin{align*}
\det(N^{\pm}) &=\det(P)\det(N_*^{\pm})\det(Q)\\
&=(\pm 1)[a_1a_2 \cdots a_n + (-1)^{n+1}(-a_n)(-a_1)(-a_2) \cdots (-a_{n-1})](\pm 1)\\
&=(\pm 1)a_1a_2 \cdots a_n [1+(-1)^{2n+1}](\pm 1)\\
&=0.
\end{align*}
\end{proof}

\begin{lemma}\label{incidencedet}
Suppose $G$ is a signed unicyclic graph on $n$ vertices with a cycle $C_k$, $3\leq k\leq n$. Then  the determinant of an oriented net incidence matrix $N^{\pm}$ of $G$ is $0$. Also the determinant of the net incidence matrix $M^{\pm}$ of $G$ is $\pm 2\i^{e^-}$ if $k$ is odd and zero otherwise, where $\mathrm e^-$ is the number of negative edges of $G$.
\end{lemma}
\begin{proof}
The proof follows from the preceding lemma when $G$ is a signed cycle. Now suppose $G$ is not a signed cycle, i.e., $G$ has pendant vertices. By successive
cofactor expansions along rows corresponding to the pendant vertices, the determinant becomes $\pm \i^{e_{C_k}^-}$ times the determinant of the signed cycle, where $\mathrm e_{C_k}^-$ is the number of negative edges in $C_k$. The rest follows from the preceding Lemma.
\end{proof}

\section{The number of both positive and negative spanning trees in a signed graph}
The matrix tree theorem can be proved using the Cauchy-Binet formula:
	\begin{theorem}[Cauchy-Binet]\label{CB}
		Let $m \leq n$. For $m\times n$ matrices $A$ and $B$, 
		\[
			\det(AB^T)=\sum_{S} \det(A(;S]) \det(B(;S]),
		\]
		where the summation runs over $\binom{n}{m}$ $m$-subsets $S$ of $\{1,2,\ldots,n\}$.
	\end{theorem}
	
In the preceding theorem, 	$A(;S]$ is the submatrix of $A$ with  the columns indexed by $S\subseteq \{1,2,\ldots,n\}$ and all the rows of $A$ (i.e., no rows deleted). In the next observation and the rest of this article, $L^{\pm}(i)$ is the submatrix of $L^{\pm}$ with row $i$ and column $i$ deleted, $N^{\pm}(i;)$ is the submatrix of $N^{\pm}$ with row $i$ deleted, and $N^{\pm}(i;S]$ is the submatrix of $N^{\pm}$ with row $i$ deleted keeping the columns indexed by $S\subseteq \{1,2,\ldots,m\}$.
	
\begin{obs}\label{CB3}
Let $G$ be a singed graph on $n\geq 2$ vertices with $m$ edges and $m \geq n-1$. Suppose $L^{\pm}$ is the net  Laplacian matrix  and  $N^{\pm}$ is an oriented net incidence matrix of $G$. Then
\begin{enumerate}
\item[(a)] $L^{\pm}(i)=N^{\pm}(i;)N^{\pm}(i;)^T$, $i=1,2,\ldots,n$, and 

\item[(b)] $\det(L^{\pm}(i))=\det(N^{\pm}(i;)N^{\pm}(i;)^T)=\sum_{S} \det(N^{\pm}(i;S])^2,$
where the summation runs over all $(n-1)$-subsets $S$ of $\{1,2,\ldots,m\}$ (by Theorem \ref{CB}).
\end{enumerate}
\end{obs}

 A subgraph $H$ of a signed graph $G$ is {\it positive} if $H$ has an even number of negative edges (i.e., the product of the signs of the edges of $H$ is positive). Similarly $H$ is {\it negative} if $H$ has an odd number of negative edges (i.e., the product of the signs of the edges of $H$ is negative).

\begin{lemma}\label{tree_det}
Let $T$ be a signed tree with at least one edge and $N^{\pm}$ be an oriented net incidence matrix of $T$. Then for all vertices $j$ of $T$,
$$\det(N^{\pm}(j;)) = \begin{cases} 
\pm 1 &\text{if } T \text{ is positive}\\
\pm \i &\text{if } T \text{ is negative.}\end{cases}$$
\end{lemma}	

\begin{proof}
We prove the statement by induction on $n$, the number of vertices of $T$. For $n = 2$,  $T$ is $P_2$ and  then $\det(N^{\pm}(j;))=\pm 1$ when $T$ is positive and $\det(N^{\pm}(j;))=\pm \i$ when $T$ is negative. Assume the statement holds for some $n \geq 2$. Let $T$ be a tree with $n+1$ vertices. Suppose $v$ is a pendant vertex of $T$ incident with the unique edge $e_{\ell} = \{ v,k \}$ of $T$. \\ 

Case 1. $T$ is a negative tree\\
Subcase (a). $e_{\ell} = \{ v,k \}$ is a negative edge\\
In this case, $T(v)$ is a positive tree. Then by the induction hypothesis, $\det(N^{\pm}(v,j;\ell))=\pm 1$ for any $j\neq v$. The $v$th row of $N^{\pm}$ has only one nonzero entry which is the $(v,\ell)$th entry and it is equal to $\pm \i$. To find $\det(N^{\pm}(j;))$, $j\neq v$, take a cofactor expansion across the $v$th row and get 
\[
\det(N^{\pm}(j;))=\pm \i \cdot \big( \pm \det(N(v,j ; \ell)) \big)=\pm \i(\pm 1)=\pm \i.
\]

Note that the $\ell$th column of $N^{\pm}(v;)$ has only one nonzero entry which is the $(k,\ell)$th entry and it is equal to $\pm \i$. To find $\det(N^{\pm}(v;))$, take a cofactor expansion across the $\ell$th column and get 
\[
\det(N^{\pm}(v;))=\pm \i \cdot \big( \pm \det(N(v,k ; \ell)) \big)=\pm \i(\pm 1)=\pm \i.
\]
 
Subcase (b). $e_l = \{ v,k \}$ is a positive edge\\
In this case, $T(v)$ is a negative tree. Then by the induction hypothesis, $\det(N^{\pm}(v,j;\ell))=\pm \i$ for any $j\neq v$. The $v$th row of $N^{\pm}$ has only one nonzero entry which is the $(v,\ell)$th entry and it is equal to $\pm 1$. To find $\det(N^{\pm}(j;))$, $j\neq v$, take a cofactor expansion across the $v$th row and get 
\[
\det(N^{\pm}(j;))=\pm 1 \cdot \big( \pm \det(N(v,j ; \ell)) \big)=\pm 1 (\pm \i)=\pm \i.
\]

Note that the $\ell$th column of $N^{\pm}(v;)$ has only one nonzero entry which is the $(k,\ell)$th entry and it is equal to $\pm 1$. To find $\det(N^{\pm}(v;))$, take a cofactor expansion across the $\ell$th column and get 
\[
\det(N^{\pm}(v;))=\pm 1 \cdot \big( \pm \det(N(v,k ; \ell)) \big)=\pm 1(\pm \i)=\pm \i.
\]

Case 2. $T$ is a positive tree\\
Subcase (a). $e_{\ell} = \{ v,k \}$ is a negative edge\\
In this case, $T(v)$ is a negative tree. Then by the induction hypothesis, $\det(N^{\pm}(v,j;\ell))=\pm \i$ for any $j\neq v$. The $v$th row of $N^{\pm}$ has only one nonzero entry which is the $(v,\ell)$th entry and it is equal to $\pm \i$. To find $\det(N^{\pm}(j;))$, $j\neq v$, take a cofactor expansion across the $v$th row and get 
\[
\det(N^{\pm}(j;))=\pm \i \cdot \big( \pm \det(N(v,j ; \ell)) \big)=\pm \i(\pm \i)=\pm 1.
\]

Note that the $\ell$th column of $N^{\pm}(v;)$ has only one nonzero entry which is the $(k,\ell)$th entry and it is equal to $\pm \i$. To find $\det(N^{\pm}(v;))$, take a cofactor expansion across the $\ell$th column and get 
\[
\det(N^{\pm}(v;))=\pm \i \cdot \big( \pm \det(N(v,k ; \ell)) \big)=\pm \i(\pm \i)=\pm 1.
\]

Subcase (b). $e_l = \{ v,k \}$ is a positive edge\\
In this case, $T(v)$ is a positive tree. Then by the induction hypothesis, $\det(N^{\pm}(v,j;\ell))=\pm 1$ for any $j\neq v$. The $v$th row of $N^{\pm}$ has only one nonzero entry which is the $(v,\ell)$th entry and it is equal to $\pm 1$. To find $\det(N^{\pm}(j;))$, $j\neq v$, take a cofactor expansion across the $v$th row and get 
\[
\det(N^{\pm}(j;))=\pm 1 \cdot \big( \pm \det(N(v,j ; \ell)) \big)=\pm 1(\pm 1)=\pm 1.
\]

Note that the $\ell$th column of $N^{\pm}(v;)$ has only one nonzero entry which is the $(k,\ell)$th entry and it is equal to $\pm 1$. To find $\det(N^{\pm}(v;))$, take a cofactor expansion across the $\ell$th column and get 
\[
\det(N^{\pm}(v;))=\pm 1 \cdot \big( \pm \det(N(v,k ; \ell)) \big)=\pm 1(\pm 1)=\pm 1.
\]
\end{proof}

\begin{lemma}\label{tree_det M(j;)}
Let $T$ be a signed tree with at least one edge and $M^{\pm}$ be the net incidence matrix of $T$. Then for all vertices $j$ of $T$,
$$\det(M^{\pm}(j;)) = \begin{cases} 
\pm 1 &\text{if } T \text{ is positive}\\
\pm \i &\text{if } T \text{ is negative.}\end{cases}$$
\end{lemma}	

\begin{proof}
The proof is similar to that for an oriented net incidence matrix.
\end{proof}

A {\it positive spanning tree} of a simple signed graph $G$ is a spanning tree of $G$ that is positive (i.e., it contains an even number of negative edges). A {\it negative spanning tree} of $G$ is a spanning tree of $G$ that is negative (i.e., it contains an odd number of negative edges).

\begin{example}
The signed paw  $G$  in Figure \ref{paw} has
one positive spanning tree $T_1$ and two negative spanning trees $T_2$ and $T_3$ (see in Figure \ref{spanning trees of paw}). 

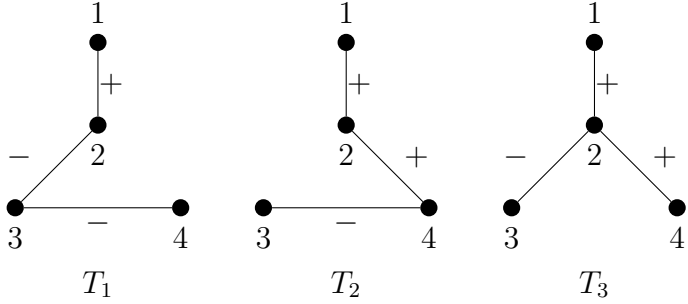
\begin{figure}
	\begin{center}
	\begin{tikzpicture}%[colorstyle/.style={circle, fill, black, scale = .5}]
	[scale=1.1,colorstyle/.style={circle, draw=black!100,fill=black!100, thick, inner sep=0pt, minimum size=2mm},>=stealth]
		\node (1) at (0,1)[colorstyle, label=above:$1$]{};
		\node (2) at (0,0)[colorstyle, label=below:$2$]{};
		\node (3) at (-1,-1)[colorstyle, label=below:$3$]{};
		\node (4) at (1,-1)[colorstyle, label=below:$4$]{};
		\node at (0,-1.5)[label=below:$T_1$]{};
		
		\draw [] (1)--(2)--(3)--(4);
		%\node at (0.85,0)[label=below:{$+$}]{};
		\node at (-0.95,0)[label=below:{$-$}]{};
		\node at (0,-.8)[label=below:{$-$}]{};
		\node at (-.25,0.5)[label=right:{$+$}]{};
		
		\node (1b) at (3,1)[colorstyle, label=above:$1$]{};
		\node (2b) at (3,0)[colorstyle, label=below:$2$]{};
		\node (3b) at (2,-1)[colorstyle, label=below:$3$]{};
		\node (4b) at (4,-1)[colorstyle, label=below:$4$]{};
		\node at (3,-1.5)[label=below:$T_2$]{};
		
		\draw [] (1b)--(2b)--(4b)--(3b);
		\node at (3.85,0)[label=below:{$+$}]{};
		%\node at (2.05,0)[label=below:{$-$}]{};
		\node at (3,-.8)[label=below:{$-$}]{};
		\node at (2.75,0.5)[label=right:{$+$}]{};
		
		\node (1c) at (6,1)[colorstyle, label=above:$1$]{};
		\node (2c) at (6,0)[colorstyle, label=below:$2$]{};
		\node (3c) at (5,-1)[colorstyle, label=below:$3$]{};
		\node (4c) at (7,-1)[colorstyle, label=below:$4$]{};
		\node at (6,-1.5)[label=below:$T_3$]{};
		
		\draw [] (1c)--(2c)--(3c);	
		\draw [] (2c)--(4c);
		\node at (6.85,0)[label=below:{$+$}]{};
		\node at (5.05,0)[label=below:{$-$}]{};
		%\node at (6,-.8)[label=below:{$-$}]{};
		\node at (5.75,0.5)[label=right:{$+$}]{};
		
	\end{tikzpicture}
	\caption{Spanning trees of the signed paw in Figure \ref{paw}}\label{spanning trees of paw}
	\end{center}
	\end{figure}
\end{example}

The following is the matrix tree theorem analog for the net Laplacian of a signed graph. 
\begin{theorem}\label{mainresult}
	Let $G$ be a simple connected signed graph on $n\geq 2$ vertices $1,2,\ldots,n$ with the net Laplacian matrix $L^{\pm}$. Then for each $i=1,2,\ldots,n$, 	$\det(L^{\pm}(i))$ is the number of positive spanning trees minus the number of negative spanning trees of $G$.
\end{theorem}
\begin{proof}
By Observation \ref{CB3}, 
	$$\det(L^{\pm}(i))=\sum_{S} \det(N^{\pm}(i;S])^2,$$
where the summation runs over all $(n-1)$-subsets $S$ of $\{1,2,\ldots,m\}$. Note that each such $S$ corresponds to a spanning subgraph $H_S$ of $G$ consisting of $n-1$ edges indexed by $S$ and possibly with some isolated vertices. Suppose $H_S$ is not a spanning tree of $G$. Then we show $\det(N^{\pm}(i;S])=0$.  Since $H_S$ is not a spanning tree of $G$, $H_S$ contains a tree component $T$ and a connected component $H'$ containing a cycle. If $i$ is not in $T$, then the rows of $N^{\pm}(i;S]$ corresponding to the vertices of $T$ are linearly dependent and consequently $\det(N^{\pm}(i;S])=0$. Suppose $i$ is in $T$ and $S'\subseteq S$ corresponds to the edges of a cycle in $H'$ on vertices $v_{i_1},v_{i_2},\ldots,v_{i_t}$, $t=|S'|$. Then $\det(N^{\pm}[i_1,i_2,\ldots,i_t;S'])=0$ by Lemma \ref{det of inc of a cycle}. Thus the columns of $N^{\pm}(i;S]$ corresponding to $S'$ are linearly dependent and consequently $\det(N^{\pm}(i;S])=0$.\\

Suppose $\mathcal T^+$ and $\mathcal T^-$ are the sets of all positive and negative spanning trees of $G$ respectively. Then
\begin{eqnarray*}
\det(L^{\pm}(i))
=\sum_{S} \det(N^{\pm}(i;S])^2
&=&\sum_{H_S\in \mathcal T^+}(\det(N^{\pm}(i;S])^2 
+\sum_{H_S\in \mathcal T^-}\det(N^{\pm}(i;S])^2\\
&=&\sum_{H\in \mathcal T^+}(\pm 1)^2 
+\sum_{H\in \mathcal T^-}(\pm i)^2 \;(\text{by Lemma \ref{tree_det}})\\
&=&|\mathcal T^+|-|\mathcal T^-|.   
\end{eqnarray*}
\end{proof}	

\begin{obs}
The preceding theorem implies the matrix tree theorem for a graph when it is considered as a signed graph with all positive edges.
\end{obs}

\begin{corollary}
Let $G$ be a simple connected signed graph on $n\geq 2$ vertices with the net Laplacian matrix $L^{\pm}$. Let $L_{|G|}$ be the Laplacian matrix of  the underlying unsigned graph $|G|$ of the signed graph $G$. Then for each $i=1,2,\ldots,n$, 
\begin{enumerate}
\item[(a)] $\frac{1}{2}\left[ \det(L_{|G|}(i))+\det(L^{\pm}(i))\right]$ is the number of positive spanning trees of $G$, and

\item[(b)] $\frac{1}{2}\left[ \det(L_{|G|}(i))-\det(L^{\pm}(i))\right]$ is the number of negative spanning trees of $G$.
\end{enumerate}
\end{corollary}

\begin{proof}
Suppose $\mathcal T^+$ and $\mathcal T^-$ are the sets of all positive and negative spanning trees of $G$ respectively. Then the proof follows from the following:
\[\det(L_{|G|}(i))=|\mathcal T^+|+|\mathcal T^-| \text{ and } \det(L^{\pm}(i))=|\mathcal T^+|-|\mathcal T^-|\]
\end{proof}

\begin{example}
The signed paw  $G$  in Figure \ref{paw} has
one positive spanning tree $T_1$ and two negative spanning trees $T_2$ and $T_3$ (see in Figure \ref{spanning trees of paw}). For each $i=1,2,3,4$, 	
\[\det(L^{\pm}(i))=1-2=-1.\]
The number of positive spanning trees of $G$ is 
\[\frac{1}{2}\left[\det(L_{|G|}(i))+\det(L^{\pm}(i))\right]
=\frac{1}{2}(3-1)=1.\]
The number of negative spanning trees of $G$ is 
\[\frac{1}{2}\left[\det(L_{|G|}(i))-\det(L^{\pm}(i))\right]
=\frac{1}{2}(3+1)=2.\]
\end{example}

\section{Signless net Laplacian}
A {\it $TU$-graph} is a graph whose connected components are trees or odd-unicyclic graphs. A {\it $TU$-subgraph} of $G$ is a subgraph of $G$ that is a $TU$-graph. We are using the same definitions for signed $TU$-graphs and signed $TU$-subgraphs of a signed graph (which is different from the definition given in \cite{Belardo1}). According to our definition before, a signed unicyclic graph is negative if it has an odd number of negative edges. A {\it negative component} of a signed $TU$-graph $G$ is a negative tree or a negative odd-unicyclic graph. The number of negative components in a signed $TU$-graph $G$ is denoted by $b^-(G)$. If $G$ is an odd-unicyclic graph, then the determinant of its net incidence matrix is $\pm 2\i^{b^-(G)}$ by Lemma \ref{incidencedet}.

\begin{figure}[!htb]
	\begin{center}
	\begin{tikzpicture}%[colorstyle/.style={circle, fill, black, scale = .5}]
	[scale=1,colorstyle/.style={circle, draw=black!100,fill=black!100, thick, inner sep=0pt, minimum size=2mm},>=stealth]
	    \node (1) at (0,2)[colorstyle, label=right:$1$]{};
		\node (2) at (0,1)[colorstyle, label=right:$2$]{};
		\node (3) at (0,0)[colorstyle, label=below:$3$]{};
		\node (4) at (-1,-1)[colorstyle, label=below:$4$]{};
		\node (5) at (1,-1)[colorstyle, label=below:$5$]{};
		\node at (0,-1.5)[label=below:$G$]{};
		\draw [] (1)--(2)--(3)--(4)--(5)--(3);
        
        \node at (0,1.5)[label=left:{$+$}]{};%e1
        \node at (0,0.5)[label=left:{$-$}]{};%e2
        \node at (-0.95,0)[label=below:{$+$}]{};%e3
        \node at (0,-.8)[label=below:{$-$}]{};%e4
        \node at (0.85,0)[label=below:{$+$}]{};%e5
		
	\end{tikzpicture}	
	\caption{A signed graph $G$}\label{extended paw}
	\end{center}
\end{figure}
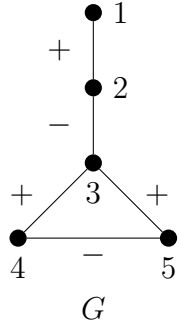

\begin{example}
Consider the signed graph $G$ in Figure \ref{extended paw}. All the spanning $TU$-subgraphs of $G$  with $4$ edges consisting of a unique tree on vertex $1$ are $H_1$, $H_2$, $H_3$, $H_4$ and $H_5$ (see Figure \ref{TU subgraphs of extendd paw}). Note that $b^-(H_3)=b^-(H_5)=1$ and $b^-(H_1)=b^-(H_2)=b^-(H_4)=0$.
\end{example}

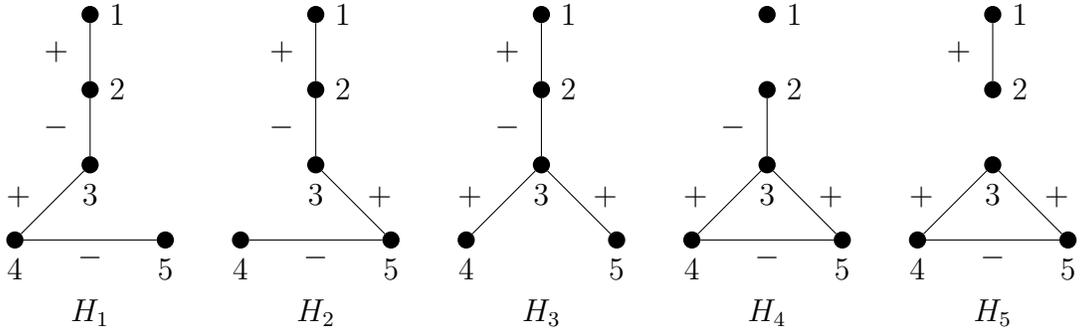
\begin{figure}
	\begin{center}
	\begin{tikzpicture}%[colorstyle/.style={circle, fill, black, scale = .5}]
	[scale=1,colorstyle/.style={circle, draw=black!100,fill=black!100, thick, inner sep=0pt, minimum size=2mm},>=stealth]
	    \node (1) at (0,2)[colorstyle, label=right:$1$]{};
		\node (2) at (0,1)[colorstyle, label=right:$2$]{};
		\node (3) at (0,0)[colorstyle, label=below:$3$]{};
		\node (4) at (-1,-1)[colorstyle, label=below:$4$]{};
		\node (5) at (1,-1)[colorstyle, label=below:$5$]{};
		\node at (0,-1.5)[label=below:$H_1$]{};
		\draw [] (1)--(2)--(3)--(4)--(5);
		\node at (0,1.5)[label=left:{$+$}]{};%e1
        \node at (0,0.5)[label=left:{$-$}]{};%e2
        \node at (-0.95,0)[label=below:{$+$}]{};%e3
        \node at (0,-.8)[label=below:{$-$}]{};%e4
        %\node at (0.85,0)[label=below:{$+$}]{};%e5
		
		\node (1b) at (3,2)[colorstyle, label=right:$1$]{};
		\node (2b) at (3,1)[colorstyle, label=right:$2$]{};
		\node (3b) at (3,0)[colorstyle, label=below:$3$]{};
		\node (4b) at (2,-1)[colorstyle, label=below:$4$]{};
		\node (5b) at (4,-1)[colorstyle, label=below:$5$]{};
		\node at (3,-1.5)[label=below:$H_2$]{};
		\draw [] (1b)--(2b)--(3b)--(5b)--(4b);
		\node at (3,1.5)[label=left:{$+$}]{};%e1
        \node at (3,0.5)[label=left:{$-$}]{};%e2
        %\node at (-0.95,0)[label=below:{$+$}]{};%e3
        \node at (3,-.8)[label=below:{$-$}]{};%e4
        \node at (3.85,0)[label=below:{$+$}]{};%e5
		
		\node (1c) at (6,2)[colorstyle, label=right:$1$]{};
		\node (2c) at (6,1)[colorstyle, label=right:$2$]{};
		\node (3c) at (6,0)[colorstyle, label=below:$3$]{};
		\node (4c) at (5,-1)[colorstyle, label=below:$4$]{};
		\node (5c) at (7,-1)[colorstyle, label=below:$5$]{};
		\node at (6,-1.5)[label=below:$H_3$]{};
		\draw [] (1c)--(2c)--(3c)--(4c);	
		\draw [] (3c)--(5c);
		\node at (6,1.5)[label=left:{$+$}]{};%e1
        \node at (6,0.5)[label=left:{$-$}]{};%e2
        \node at (5.05,0)[label=below:{$+$}]{};%e3
        %\node at (6,-.8)[label=below:{$-$}]{};%e4
        \node at (6.85,0)[label=below:{$+$}]{};%e5
		
		\node (1d) at (9,2)[colorstyle, label=right:$1$]{};
		\node (2d) at (9,1)[colorstyle, label=right:$2$]{};
		\node (3d) at (9,0)[colorstyle, label=below:$3$]{};
		\node (4d) at (8,-1)[colorstyle, label=below:$4$]{};
		\node (5d) at (10,-1)[colorstyle, label=below:$5$]{};
		\node at (9,-1.5)[label=below:$H_4$]{};
		\draw [] (2d)--(3d)--(4d)--(5d)--(3d);
		%\node at (9,1.5)[label=left:{$+$}]{};%e1
        \node at (9,0.5)[label=left:{$-$}]{};%e2
        \node at (8.05,0)[label=below:{$+$}]{};%e3
        \node at (9,-.8)[label=below:{$-$}]{};%e4
        \node at (9.85,0)[label=below:{$+$}]{};%e5
		
		\node (1e) at (12,2)[colorstyle, label=right:$1$]{};
		\node (2e) at (12,1)[colorstyle, label=right:$2$]{};
		\node (3e) at (12,0)[colorstyle, label=below:$3$]{};
		\node (4e) at (11,-1)[colorstyle, label=below:$4$]{};
		\node (5e) at (13,-1)[colorstyle, label=below:$5$]{};
		\node at (12,-1.5)[label=below:$H_5$]{};
		\draw [] (3e)--(4e)--(5e)--(3e);
		\draw [] (1e)--(2e);
		\node at (12,1.5)[label=left:{$+$}]{};%e1
        %\node at (0,0.5)[label=left:{$-$}]{};%e2
        \node at (11.05,0)[label=below:{$+$}]{};%e3
        \node at (12,-.8)[label=below:{$-$}]{};%e4
        \node at (12.85,0)[label=below:{$+$}]{};%e5
	
	\end{tikzpicture}
	\caption{Spanning $TU$-subgraphs of $G$ (in Figure \ref{extended paw}) with $4$ edges containing a unique tree on vertex $1$}\label{TU subgraphs of extendd paw}
	\end{center}
\end{figure}

\begin{theorem}\label{mainlemma}
	Let $G$ be a simple connected signed graph on $n\geq 2$ vertices and $m$ edges with the net incidence matrix $M^{\pm}$. Let $i$ be an integer from $\{1,2,\ldots,n\}$. Let  $S$  be an $(n-1)$-subset of $\{1,2,\ldots,m\}$ and $H$ be a spanning subgraph of $G$ with edges indexed by $S$ and possibly with some isolated vertices. 
\begin{enumerate}
\item[(a)] If $H$ is not a spanning $TU$-subgraph of $G$, then $\det(M^{\pm}(i;S])=0$.

\item[(b)] Suppose $H$ is a spanning $TU$-subgraph of $G$ consisting of a unique tree $T$ and $c$  odd-unicyclic graphs $U_1, U_2, \ldots, U_c$.  
    \begin{enumerate}
        \item[(i)] If $i$ is a vertex of $U_j$ for some $j=1,2,\ldots,c$, then $\det(M^{\pm}(i;S]) =0$. 
        
        \item[(ii)] If $i$ is a vertex of $T$, then  $\det(M^{\pm}(i;S]) =\pm 2^c\i^{b^-(H)}$.
        
    \end{enumerate}

\end{enumerate}
\end{theorem}

\begin{proof}
(a) Suppose $H$ is not a  $TU$-graph on $n$ vertices and $n-1$ edges. Then $n\geq 5$ and  $H$ has at least two tree components and a component containing at least two cycles. Suppose $T$ is a tree component of $H$ that does not contain vertex $i$. If $T$ consists of just one vertex, then the corresponding row in $M^{\pm}(i;)$ is a zero row giving $\det(M^{\pm}(i;))=0$. Now suppose  $T$ has at least two vertices.  
Consider  the square submatrix $M'$ of $M^{\pm}(i;)$ with rows corresponding to the vertices of $T$ and columns corresponding to the edges of $T$. Since the row rank of $M'$ is less than or equal to the number of columns of $M'$, the rows of $M'$ and consequently of $M^{\pm}(i;)$ corresponding to the vertices of $T$ are linearly dependent. Thus $\det(M^{\pm}(i;))=0$.\\

(b) If $i$ is a vertex of $U_j$ for some $j=1,2,\ldots,c$, then the columns of the submatrix of $M^{\pm}(i;S]$ corresponding to $U_j$ are linearly dependent which implies $\det(M^{\pm}(i;S])=0$. If $i$ is a vertex of the positive (respectively negative) tree $T$, then $M^{\pm}(i;S]$ is a direct sum of incidence matrices of odd-unicyclic graphs $U_1, U_2, \ldots , U_c$  and the incidence matrix of the tree $T$ with one row deleted and consequently $\det(M^{\pm}(i;S])=(\pm 2^c\i^{b^-(H)})\cdot (\pm 1)= \pm 2^c\i^{b^-(H)}$ (respectively $\det(M^{\pm}(i;S])=(\pm 2^c\i^{-1+b^-(H)})\cdot (\pm \i)= \pm 2^c\i^{b^-(H)}$) by Lemmas \ref{incidencedet} and \ref{tree_det M(j;)}.
\end{proof}

The following is the matrix tree theorem analog for the signless net Laplacian of a signed graph.
\begin{theorem}
Let $G$ be a simple connected signed graph on $n\geq 2$ vertices $1,2,\ldots,n$ with the signless net Laplacian matrix $Q^{\pm}$. Then for each $i=1,2,\ldots,n$, \[\det(Q^{\pm}(i))=
\sum_{H\in \mathcal U_e}4^{c(H)}
-\sum_{H\in\mathcal U_o} 4^{c(H)},\]
where $\mathcal U_e$ is the set of all spanning $TU$-subgraphs $H$ of $G$ with $n-1$ edges and an even number of negative components  consisting of a unique tree on vertex $i$ and  $c(H)$ odd-unicyclic graphs and 
$\mathcal U_o$ is the set of all spanning $TU$-subgraphs $H$ of $G$ with $n-1$ edges and an odd number of negative components  consisting of a unique tree on vertex $i$ and  $c(H)$ odd-unicyclic graphs.
\end{theorem}

\begin{proof}
First note that
\[\det(Q^{\pm}(i))=\sum_{S} \det(M^{\pm}(i;S])^2,\]
where the summation runs over all $(n-1)$-subsets $S$ of $\{1,2,\ldots,m\}$. Note that each such $S$ corresponds to a spanning subgraph $H_S$ of $G$  consisting of $n-1$ edges indexed by $S$ and possibly with some isolated vertices. Note by Theorem \ref{mainlemma}, $\det(M^{\pm}(i;S]) \neq 0$ only when $H_S$ is a spanning $TU$-subgraph of $G$ consisting of a unique tree $T$ containing vertex $i$ and $c(H_S)$  odd-unicyclic graphs. Then
\begin{eqnarray*}
\det(Q^{\pm}(i))
=\sum_{S} \det(M^{\pm}(i;S])^2
&=&\sum_{H_S\in \mathcal U_e}(\det(M^{\pm}(i;S])^2 
+\sum_{H_S\in \mathcal U_o}\det(M^{\pm}(i;S])^2\\
&=&\sum_{H_S\in \mathcal U_e}(\pm 2^{c(H_S)} \i^{b^-(H_S)})^2 
+\sum_{H_S\in\mathcal U_o}(\pm 2^{c(H_S)}\i^{b^-(H_S)} )^2 \;(\text{by Theorem \ref{mainlemma}})\\
&=&\sum_{H_S\in \mathcal U_e}(\pm 2^{c(H_S)})^2 
+\sum_{H_S\in\mathcal U_o}(\pm \i 2^{c(H_S)})^2 \\
&=&\sum_{H\in \mathcal U_e}4^{c(H)}
-\sum_{H\in \mathcal U_o} 4^{c(H)}.   
\end{eqnarray*}
\end{proof}

\begin{obs}
The preceding theorem implies the matrix tree theorem for the signless Laplacian of a graph (see \cite[Theorem 2.9]{Mallik}) when it is considered as a signed graph with all positive edges.
\end{obs}

\begin{example}
For the signed graph $G$ in Figure \ref{extended paw},  $\mathcal U_e=\{H_1,H_2,H_4\}$ and $\mathcal U_o=\{H_3,H_5\}$   corresponding to vertex $1$ (see Figure \ref{TU subgraphs of extendd paw}). 
\begin{eqnarray*}
\det(Q^{\pm}(1))
=\sum_{H\in \mathcal U_e}4^{c(H)}
-\sum_{H\in\mathcal U_o} 4^{c(H)}
&=&\left(4^{c(H_1)}+4^{c(H_2)}+4^{c(H_4)}\right)
-\left(4^{c(H_3)}+4^{c(H_5)}\right)\\
&=&\left(4^0+4^0+4^1\right)
-\left(4^0+4^1\right)\\
&=&6-5\\
&=&1 \end{eqnarray*}
\end{example}

\section{Determinant of a signless net Laplacian}
Since the row-sums of the net Laplacian $L^{\pm}$ of a signed graph $G$ are zeros,  $0$ is an eigenvalue of $L^{\pm}$ and consequently $\det(L^{\pm})=0$. In this section we investigate $\det(Q^{\pm})$, the determinant of the signless net Laplacian $Q^{\pm}$ of a signed graph $G$.

\begin{lemma}\label{oddunicycle}
Let $G$ be a simple signed graph on $n$ vertices and $m\geq n$ edges with the net incidence matrix $M^{\pm}$. Let  $S$  be a $n$-subset of $\{1,2,\ldots,m\}$ and $H$ be a spanning subgraph of $G$ with edges indexed by $S$ and possibly some isolated vertices.  
\begin{enumerate}
\item[(a)]  If one of the connected components of $H$ is not an odd-unicyclic graph, then $\det(M^{\pm}[S])=0$.

\item[(b)] If $H$ consists of $k$ odd-unicyclic connected components and even of them are negative, then $\det(M^{\pm}[S])=\pm 2^k$. If $H$ consists of $k$ odd-unicyclic connected components and odd of them are negative, then $\det(M^{\pm}[S])=\pm 2^k\i$.
\end{enumerate}
\end{lemma}

\begin{proof}
(a) Suppose one of the connected components of $H$ is not an odd-unicyclic graph. Then $H$ has a connected component that is a tree or an even-unicyclic graph.  Then by Lemma \ref{incidencedet}, $\det(M^{\pm}[S])=0$ when $H$ has an even-unicyclic component. When $H$ has a tree component $T$, the rows of $M^{\pm}[S]$ corresponding to the vertices of $T$ are linearly dependent and consequently $\det(M^{\pm}[S])=0$. \\

(b) Suppose $H$ consists of $k$ odd-unicyclic components. Since $M^{\pm}[S]$ is a direct sum of the net incidence matrices of $k$ odd-unicyclic graphs, then $\det(M^{\pm}[S])=\pm 2^k\i^{b^-(H)}$ by Lemma \ref{incidencedet}. Thus $\det(M^{\pm}[S])$ is $\pm 2^k\i$ when $b^-(H)$ is odd and $\pm 2^k$ otherwise.
\end{proof}

By Theorem \ref{CB} and \ref{oddunicycle}, we have the following theorem.

\begin{theorem}\label{detQpm}
Let $G$ be a simple signed graph on $n$ vertices with signless net Laplacian matrix $Q^{\pm}$. Then 
$$\det(Q^{\pm})=\sum_{H\in \mathcal U_e^n} 4^{c(H)}-\sum_{H\in \mathcal U_o^n} 4^{c(H)},$$
where $\mathcal U_e^n$ is the set of all spanning $TU$-subgraphs $H$ of $G$ with $n$ edges consisting of $c(H)$ odd-unicyclic graphs including an even number of negative odd-unicyclic graphs and 
$\mathcal U_o^n$ is the set of all spanning $TU$-subgraphs $H$ of $G$ with $n$ edges  consisting of $c(H)$ odd-unicyclic graphs including an odd number of negative odd-unicyclic graphs.
\end{theorem}

\begin{proof}
By Theorem \ref{CB},
$$\det(Q^{\pm})=\det(M^{\pm} (M^{\pm})^T)=\sum_{S} \det(M^{\pm}(;S])^2,$$
where the summation runs over all $n$-subsets $S$ of $\{1,2,\ldots,m\}$. By Lemma \ref{oddunicycle}, 
\begin{eqnarray*}
\det(Q^{\pm})=\sum_{S} \det(M^{\pm}(;S])^2 &=& \sum_{H\in \mathcal U_e^n} (\pm 2^{c(H)})^2-\sum_{H\in \mathcal U_o^n} (\pm 2^{c(H)}\i)^2\\
&=& \sum_{H\in \mathcal U_e^n} 4^{c(H)}-\sum_{H\in \mathcal U_o^n} 4^{c(H)}.
\end{eqnarray*}
\end{proof}

\begin{example}
The signed graph $G$ in Figure \ref{extended paw} is a $TU$-graph, in particular, an odd-unicyclic graph that is positive. Then $\mathcal U_e^5=\{G\}$ and $\mathcal U_o^5=\varnothing$. Thus
\[\det(Q^{\pm})
=\sum_{H\in \mathcal U_e^5}4^{c(H)}
-\sum_{H\in\mathcal U_o^5} 4^{c(H)}
=4^{c(G)}-0\\
=4^1\\
=4.\]
\end{example}

\section{Open problems}
We end this article by posing some relevant open problems.

\begin{question}
Find combinatorial formulas for the coefficients of the characteristic polynomial of the signless net Laplacian matrix $Q^{\pm}$ of a signed graph.
\end{question}
Suppose the characteristic polynomial of signless net Laplacian $Q^{\pm}$ of a simple connected signed graph $G$ on $n$ vertices and $m\geq n$ edges is 
\[P_{Q^{\pm}}(x)=\det(x I_n-Q^{\pm})=x^n+\sum_{i=1}^n a_i x^{n-i}.\]
For $i=1,2,\ldots,n$, $a_i$ will be obtained in terms of spanning $TU$-subgraphs of $G$ with $i$ edges. 

\begin{question}
Characterize the signed graphs whose signless net Laplacian matrix is not invertible. 
\end{question}
One example is a bipartite graph with all positive edges. Using $\det(Q^{\pm})=0$, Theorem \ref{detQpm} gives the following necessary and sufficient condition:
\[\sum_{H\in \mathcal U_e^n} 4^{c(H)}=\sum_{H\in \mathcal U_o^n} 4^{c(H)}\]

\begin{question}
Find the multiplicity of the eigenvalue $0$ of the signless net Laplacian matrix of a signed graph when it is not invertible. 
\end{question}
Examples suggest that the answer may have a connection with the balancedness of cycles in spanning $TU$-subgraphs with $n$ edges. Observation \ref{obs eigenvector} regarding eigenvectors corresponding to the eigenvalue $0$ may be helpful.\\

As pointed out by the reviewer, the net incidence matrix $M^{\pm}$ of a signed graph $G$ can be obtained from the incidence matrix $M$ of its underlying unsigned graph $|G|$ by multiplying its columns corresponding to the negative edges by $\i$. Therefore there are different ways of proving some of the results in this article. \\

\noindent {\bf Acknowledgments}\\
The author is indebted to Keivan Hassani Monfared for the idea of introducing the imaginary number $\i$ in incidence matrices of signed graphs. The author would like to thank Zoran Stani\' c and Thomas Zaslavsky for their comments on some of the results in this article. The author would also like to thank the anonymous reviewer and the handling editor Kevin Vander Meulen for their valuable comments.

\end{document}